\documentclass[11pt,a4paper,reqno]{amsart}
 \usepackage{amsaddr}
\usepackage{amscd,amsmath}
\usepackage{amssymb,amsmath,latexsym,color,enumitem,tikz,dsfont,tikz-cd}
\usepackage[all]{xypic}       % Package XYpic of commutative diagram macros
\usepackage[varg]{pxfonts}
\usepackage{tabu,adjustbox,booktabs}
\usepackage{graphicx}

\usepackage{hyperref}

\usetikzlibrary{decorations.pathmorphing,shapes}

\usetikzlibrary{decorations.markings}
\tikzset{negated/.style={
        decoration={markings,
            mark= at position 0.5 with {
                \node[transform shape] (tempnode) {$\backslash$};
            }
        },
        postaction={decorate}
    }
}

 \newcommand{\newdownarrow}{{{\rlap{$\ $}\hbox{$\downarrow$}}}}%%
 \newcommand{\newuparrow}{{{\rlap{$\ $}\hbox{$\uparrow$}}}}%%
 \newcommand{\twoheaddownarrow}{{\rlap{\rlap{$\ $}\raise .25ex\hbox{$\downarrow$}}\raise-.25ex\hbox{$\downarrow$}}}%%
 \newcommand{\twoheaduparrow}{{\rlap{\rlap{$\ $}\raise .25ex\hbox{$\uparrow$}}\raise-.25ex\hbox{$\uparrow$}}}%%

\newcommand{\set}[1]{\{\,#1\,\}}

\newcommand{\tbigwedge}{\mathop{\textstyle \bigwedge }}%%
\newcommand{\tbigcap}{\mathop{\textstyle \bigcap }}%%
\newcommand{\tbigcup}{\mathop{\textstyle \bigcup }}%%

\newcommand{\tbigvee}{\mathop{\textstyle \bigvee }}%%
 \newcommand{\cat}[1]{\ensuremath{\mathsf{#1}}} % Category
 
\makeatletter
\newcommand*{\@old@slash}{}\let\@old@slash\slash
\def\slash{\relax\ifmmode\delimiter"502F30E\mathopen{}\else\@old@slash\fi}
\makeatother

\def\N{\mathbb{N}}

\def\SS{\mathsf{S}}

\def\cl{\mathfrak{c}}
\def\op{\mathfrak{o}}
\def\bl{\mathfrak{b}}
\newtheorem{theorem}{Theorem}[section]
\newtheorem{proposition}[theorem]{Proposition}

\newtheorem{lemma}[theorem]{Lemma}
\newtheorem{corollary}[theorem]{Corollary}

\theoremstyle{definition}

\theoremstyle{remark}
\newtheorem{remark}[theorem]{Remark}
\newtheorem{remarks}[theorem]{Remarks}

\title[The DeMorganization of a locale]{The DeMorganization of a locale}
\author{Igor Arrieta}
\address{School of Computer Science \\
University of Birmingham, B15 2TT, Birmingham, UK}
\email{i.arrietatorres@bham.ac.uk}
\thanks{The author acknowledges support from the Basque Government (grant IT1483-22 and a postdoctoral fellowship of the Basque Government, grant POS-2022-1-0015).}

\begin{document}

\maketitle

\begin{abstract}
In 2009, Caramello proved that each topos has a largest dense subtopos whose internal logic satisfies De Morgan law (also known as the law of the weak excluded middle). This finding implies that every locale has a largest dense extremally disconnected sublocale, referred to as its DeMorganization. In this paper, we take the first steps in exploring the DeMorganization in the localic context, shedding light on its geometric nature by showing that it is always a fitted sublocale and by providing a concrete description. Explicit examples of DeMorganizations for toposes that do not satisfy De Morgan law are rather difficult to find. We present a contribution in that direction, with the main result of the paper  showing that for any metrizable locale (without isolated points), its DeMorganization coincides with its Booleanization.
This, in particular, implies that any extremally disconnected metric locale (without isolated points) must be Boolean, generalizing a well-known result for topological spaces to the localic setting.
\end{abstract}

\section{Introduction}
Regular subobjects in the category of locales are known as sublocales or point-free subspaces. Given a topological space $X$,  there are typically more sublocales in its frame of opens $\Omega(X)$ than subspaces in $X$. In fact, some of these sublocales are genuinely point-free in the sense that they may not have any point at all. An important example of this situation is the \emph{Booleanization} of a locale $L$. This can be characterized either as 
\begin{enumerate}[label=\textup{(\arabic*)},leftmargin=2.0\parindent]
\item The least dense sublocale of $L$ (Isbell's density theorem --- see \cite{ATOM}), or
\item The largest Boolean dense sublocale of $L$,
\end{enumerate}
which together imply that the Booleanization is in fact the unique Boolean dense sublocale of $L$. This has of course no counterpart in the category of topological spaces and, moreover,
for a large class of spaces, including the Hausdorff ones without isolated points, the Booleanization does not contain any point at all.  This difference between classical and localic topology is  highlighted by Johnstone in \cite{pointles}: "\emph{If you want to "sell" locale theory to a classical topologist, it's a good idea to begin by asking him to imagine a world in which an intersection of dense subspaces would always be dense; once he contemplated some of the wonderful consequences that would flow from that result, you can tell him that that world is exactly the category of locales.}"

On the other hand, Caramello \cite{DM} proved that every topos has a largest dense De Morgan subtopos (see also \cite{FIELDS}), called its DeMorganization. By taking into account the well-known correspondence between sublocales of a locale and  subtoposes of the associated localic topos, it follows immediately that every locale has a largest dense and extremally disconnected sublocale, also referred to as its DeMorganization  in the sequel.

To the best of our knowledge, this construction has  not been studied in the localic context, with the exception of Johnstone's talk at TACL 2022 in Coimbra \cite{johtalk}; and it is the aim of this paper to fill that gap.  We show that the DeMorganization of a locale $L$	admits a simple and transparent description when sublocales are regarded as specific subsets of $L$ (cf. \cite{PP12}), which yields the conclusion that it is always a fitted sublocale of $L$ (the point-free counterpart of the notion of saturated subspace of a space).

 Motivated by the classical fact that a metrizable extremally disconnected space must be discrete (see e.g. \cite[14N(2)]{gilman}), we then study the DeMorganization of a metric locale. We point out here that the class of metrizable \emph{locales} is fairly more general than that of metrizable \emph{spaces} (as opposed to the case of \emph{complete} metric locales, which are always spatial under classical logic --- cf. \cite{ATOM}).

In fact, it is not easy in general to compute the DeMorganization of a given topos. As Caramello and Johnstone note, "\emph{Explicit examples of DeMorganizations, for toposes which do not satisfy De Morgan’s law, seem to be rather hard to find}" \cite[p.~2145]{FIELDS}.

The main result of that paper presents a contribution in that direction by proving that for any metrizable locale (without isolated points), its DeMorganization coincides with its Booleanization. In particular, an extremally disconnected metrizable locale (without isolated points) must be Boolean.

The paper is organized as follows. In Section~\ref{sect2} we provide the necessary preliminaries on localic topology, and in Section~\ref{sect3} we make a few observations about the Booleanization of a locale, which are useful to the remainder of the paper. In Section~\ref{sect4} we explore the DeMorganization and provide a concrete description of it. Finally, in Section~\ref{sect5} we study the DeMorganization of metrizable locales.\\[2mm]
\noindent\textbf{Acknowledgement.} I am grateful to Professor Peter Johnstone for useful discussions on this topic following his TACL talk. I also thank the anonymous referee for their suggestions, which have improved the presentation of the paper.

\section{Preliminaries}\label{sect2}
Our notation and terminology regarding the categories of frames and locales will be that of \cite{PP12} (cf. also \cite{stone}). The Heyting operator in a frame $L$, right adjoint to the meet operator, will be denoted by $\to$; for each $a\in L$, $a^*=a\to 0$ is the \emph{pseudocomplement} of $a$. An element $a\in L$ is said to be \emph{dense} if $a^*=0$, and a frame homomorphism $h\colon L\to M$ is \emph{dense} if $h(a)=0$ implies $a=0$ for all $a\in L$.

\subsection{Some Heyting rules}
For the reader's convenience, we list here some of the properties satisfied by the Heyting operator in a frame $L$. For any $a,b,c\in L$ and $\{a_i\}_{i\in I}\subseteq L$, the following hold:
\begin{enumerate}[label=\textup{(H\arabic*)},leftmargin=2.0\parindent]
\item \label{H1} $1\to a=a$\textup;
\item \label{H2} $a\leq b$ if and only if $a\to b=1$\textup;
\item \label{H3} $a\leq b\to a$\textup;
\item \label{H4} $a\to b=a\to (a\wedge b)$\textup;
\item \label{H5} $a\wedge (a\to b)=a\wedge b$\textup;
\item \label{H6} $a\wedge b=a\wedge c$ if and only if $a\to b=a\to c$\textup;
\item \label{H7} $(a\wedge b)\to c=a\to (b\to c)= b\to (a\to c)$\textup;
\item \label{H8} $a=(a\vee b)\wedge (b\to a)$\textup;
\item \label{H9} $a\leq (a\to b)\to b$\textup;
\item \label{H10} $((a\to b)\to b)\to b=a\to b$\textup;
\item \label{H11} $(\tbigvee_i a_i)\to b=\tbigwedge_i (a_i\to b)$\textup;
\item\label{H12} $b\to (\tbigwedge_i a_i)=\tbigwedge_i(b\to a_i)$\textup.
\end{enumerate}

\subsection{Sublocales}\label{subsect.sublocales}

A \emph{sublocale} of a locale $L$ is a subset $S\subseteq L$ closed under arbitrary meets such that
\[\forall a\in L,\ \ \ \forall s\in S,\  \ \ a\to s\in S.\]
These are precisely the subsets of $L$ for which the embedding $j_S\colon S\hookrightarrow L$ is a morphism of locales, where by a morphism of locales we mean a map that is the right adjoint of a frame homomorphism. Sublocales of $L$ are in one-to-one correspondence with the regular subobjects (equivalently, extremal subobjects) of $L$ in \cat{Loc}. If $\nu_S$ denotes the associated frame surjection, then for any $a\in L$ and $s\in S$ one has
\begin{equation}\label{LocM}\tag{LM}
\nu_S(a)\to s=a\to s.\end{equation}

The system $\SS  (L)$ of all sublocales of $L$, partially ordered by inclusion, is a coframe \cite[Theorem~III.3.2.1]{PP12}, that is, its dual lattice is a frame.  Infima and suprema are given by
\[
\tbigwedge_{i\in I}S_i=\tbigcap_{i\in I}S_i, \quad \tbigvee_{i\in I}S_i=\set{\tbigwedge M\mid M\subseteq\tbigcup_{i\in I} S_i}.
\]
The least element is the sublocale $\mathsf{O}=\{1\}$ and the greatest element is the entire locale $L$. For any $a\in L$, the sublocales
\[
\mathfrak{c}_L(a)=\newuparrow  a=\set{b\in L\mid b\ge a}\ \text{ and }\ \mathfrak{o}_L(a)=\set{a\to b\mid b\in L}
\]
are the \emph{closed} and \emph{open} sublocales of $L$, respectively (that we shall denote simply by $\mathfrak{c}(a)$ and $\mathfrak{o}(a)$ when there is no danger of confusion). For each $a\in L$, $\mathfrak{c}(a)$ and $\mathfrak{o}(a)$ are
complements of each other in $\SS(L)$
and satisfy the expected identities
\begin{equation*}{\label{identities.basic}}
\tbigcap_{i\in I} \mathfrak{c}(a_i)=\cl(\tbigvee_{i\in I} a_i),\quad \cl(a)\vee\cl(b)=\cl(a\wedge b),
\end{equation*}
\[\tbigvee_{i\in I}\op(a_i)=\op(\tbigvee_{i\in I} a_i) \quad\mbox{ and }\quad \op(a)\cap \op(b)=\op(a\wedge b).
\]

Given a sublocale $S$ of $L$, its \emph{closure}, denoted by $\overline{S}$, is the smallest closed sublocale containing it. In this context, the formula $\overline{S}=\cl(\tbigwedge S)$ holds. A sublocale $S$ is \emph{dense} if $\overline{S}=L$ --- i.e. iff $\tbigwedge S=0$. Therefore, since sublocales are closed under the Heyting operator,  in a dense sublocale pseudocomplementation is inherited from the ambient locale.
Note also that the sublocale $\op(a)$ is dense iff $a$ is a dense element (that is $a^*=0$).

A sublocale is said to be \emph{fitted} if it is the intersection of all the open sublocales containing it. It is easy to check that a sublocale is fitted if and only if it is an intersection of a family of open sublocales. Hence, it constitutes a (non-conservative) point-free extension of the notion of saturated subspace. If $S$ is a sublocale  of $L$, the \emph{fitting} of $S$ is the intersection of the open sublocales containing $S$, that is, the smallest fitted sublocale containing $S$ (cf. \cite{Arrieta,otherclosure} for more information).

A \emph{point} of a locale is an element $p\in L$ such that $p\ne 1$ and $a\wedge b \leq p$ implies $a\leq p$ or $b\leq p$. Then, the set $\bl(p)=\{1,p\}$ is a sublocale. We say that $p$ is \emph{isolated} if the sublocale $\bl(p)$ is open.

\subsection{Extremally disconnected locales}\label{ext}

A locale is said to be \emph{extremally disconnected} \cite{PJmor,Arrieta0} if the second De Morgan law is satisfied in $L$ --- i.e. if
$$(a\wedge b)^* = a^* \vee b^* \qquad \text{ for all } a,b\in L, $$
or equivalently if the relation
$$a^*\vee a^{**}=1 \qquad \text{ for all }a\in L$$
holds in $L$. It is easy to show that a locale $L$ is extremally disconnected if and only if the closure of any open sublocale is open. Therefore, extremal disconnectedness is a conservative extension of the homonymous topological property. Furthermore, the internal logic of a topos of sheaves on a locale $L$ satisfies De Morgan law  if and only if $L$ is extremally disconnected \cite{SKET}.

\subsection{Metric locales}
Since we shall make use of the notion of metric locale via diameters, we briefly recall it here (see \cite{diam1,PP12} for more information). Let $L$ be a locale. A \emph{diameter} on $L$ is a map $d\colon L\to [0,+\infty]$ satisfying the following properties:
\begin{enumerate}[label=\textup{(D\arabic*)},leftmargin=2.0\parindent]
\item \label{D1} $d(0)=0$,
\item \label{D2} $a\leq b$ implies $d(a)\leq d(b)$,
\item \label{D3} $a\wedge b\ne 0$ implies $d(a\vee b)\leq d(a)+d(b)$,
\item \label{D4} For every $\varepsilon>0$, the set $$U_{\varepsilon}^d =\{ a\in L\mid d(a)<\varepsilon\}$$ is a cover of $L$.
\end{enumerate}
The pair $(L,d)$ is called a \emph{pre-metric locale}. Furthermore,  the diameter $d$ is said to be \emph{admissible} if for every $a\in L$, 
$$a=\tbigvee\{ b\in L\mid b\lhd_\varepsilon a,\ \varepsilon>0\}$$
where $b\lhd_\varepsilon a$ means that  for all $c\in L$ with $d(c)<\varepsilon$, $c\wedge b\ne 0$ implies $c\leq  a$. Note that, by virtue of \ref{D4}, the relation $b\lhd_\varepsilon a$ implies $b\leq a$.
If $d$ is an admissible diameter on a pre-metric locale $L$, then the pair $(L,d)$ is said to be a \emph{metric locale}. If $(L,d)$ is a metric locale, then $L$ is \emph{regular}, that is, for any $a\in L$, the relation $a=\tbigvee\set{b\in L\mid b\prec a}$ holds, where $b\prec a$ means $b^*\vee a=1$ (in fact, $L$ is even \emph{completely regular}, but we shall not need this fact).

\section{The Booleanization of a locale}\label{sect3}

Let $L$ be a locale. Recall that its \emph{Booleanization} can be described as the least dense sublocale; or equivalent as its largest (and thus unique) Boolean dense sublocale (see \cite{BP1} for a detailed study). Explicitly, when we look at a sublocale  of $L$ as a subset of $L$ (cf. Subsection~\ref{subsect.sublocales}),  one has
 $$B_L=\{ a\in L\mid a=a^{**}\}=\{a^*\mid a\in L\}.$$ With regard to the geometric nature of the Booleanization, we point out that the fact that $B_L$ is always fitted might be well known; but we have not found this result in the literature:

\begin{proposition}\label{boolean1}
For any locale $L$, one has
$$B_L=\tbigcap \{\op(a) \mid a\in L,\,a^*=0\}$$
--- i.e. $B_L$ is the intersection of all the dense open sublocales. In particular, $B_L$ is fitted.
\end{proposition}

\begin{proof}
First, let us verify the  inclusion ``$\subseteq$''. Let $a\in L$ with $a=a^{**}$ and $b\in L$ with $b^*=0$. We have to check that $b\to a\leq a$. But $b\to a=b\to a^{**}= (b\wedge a^*)^*$ by \ref{H7}, so $b\to a\leq a$ if and only if $b\to a\leq a^{**}$ if and only if  $(b\to a)\wedge a^*=0$ if and only if  $a^*\leq (b\to a)^{*}=(b\wedge a^{*})^{**}$, and the last relation holds because double pseudocomplement commutes with finite meets and $b$ is dense.

For the reverse inclusion, let $b\in L$ with $b\in \op(a)$ for any dense $a\in L$. Choose $a=b\vee b^*$, which is dense --- i.e. $a^*=0$. Then,  $b\in \op(b\vee b^*)$ and so by \ref{H11}, \ref{H2} and \ref{H4} one has $(b\vee b^*)\to b=b^*\to b= b^*\to (b^*\wedge b)=b^*\to 0=b^{**}\leq b$, and hence $b\in B_L$.
\end{proof}

\noindent{\textbf{Observation.}} Note that an element $a\in L$ is dense if and only if it is of the form $b\vee b^*$ for some $b\in L$. Indeed, a dense element is of that form by choosing $b=a$; and conversely  $(b\vee b^*)^*=b^*\wedge b^{**}=0$. Hence, we alternatively have
$$B_L=\tbigcap_{a\in L}\op(a\vee a^*).$$
This alternative expression will make the comparison with the DeMorganization transparent.

\section{The DeMorganization of a locale}\label{sect4}

Caramello \cite{DM} showed that each topos has a largest dense De Morgan subtopos. When applied to a topos of sheaves over a locale, it follows immediately that every locale has a largest dense extremally disconnected sublocale.  The first steps to study the localic construction were taken in \cite{johtalk}, where a proof  of its existence was given based on fiberwised closed nuclei. Here, we give a simpler and transparent proof based on sublocales, which highlights the geometric nature of the construction.

A frame homomorphism $h\colon L\to M$ is said to be \emph{nearly open} if it commutes with pseudocomplementation --- i.e. if $h(a^*)=h(a)^*$ for any $a\in L$. The following is well known (see \cite{nearlyopen}), but we include a short proof for the sake of completeness:

\begin{lemma}
A dense surjection of frames is nearly open.
\end{lemma}

\begin{proof}
Let $h\colon L\to M$ be a dense  surjection of frames and $a\in L$. The inequality $h(a^*)\leq h(a)^*$ holds because $h$ preserves binary meets and the least element. For the reverse one, since $h$ is surjective, let  $b\in L$ with $h(a)^*=h(b)$.  Then $h(a\wedge b)=h(a)\wedge h(b)=0$ and by density $a\wedge b=0$, so $b\leq a^*$ and $h(a)^*=h(b)\leq h(a^*)$, as required.
\end{proof}

We are now ready to prove the main of result of this section:

\begin{proposition}\label{dm}
For any locale $L$, the sublocale $$M_L:=\tbigcap_{a\in L}\op(a^*\vee a^{**})$$
is the largest dense extremally disconnected sublocale of $L$.
\end{proposition}

\begin{proof}
It is clear that $M_L$ is a dense sublocale because each $a^*\vee a^{**}$ is dense, and intersections of dense sublocales are dense. Hence we shall denote by $(-)^*$ the pseudocomplement in either $M_L$ or $L$ without danger of confusion.
Let us now prove that it is extremally disconnected. Denote by $\nu_{M_L}$ the frame surjection associated to $M_L$, and let $b\in M_L$. Then, by the previous lemma, $$b^*\vee^{M_L} b^{**}= \nu_{M_L}(b^*\vee b^{**})\geq \nu_b(b^*\vee b^{**})=(b^{*}\vee b^{**})\to (b^{*}\vee b^{**}) = 1$$
where  $\nu_b$ denotes the frame surjection corresponding to $\op(b^*\vee b^{**})$ (note that since $M_L\subseteq \op(b^*\vee b^{**})$, one has $\nu_b\leq \nu_{M_L}$ pointwisely, when they are seen as maps $L\to L$); and $\vee^{M_L}$ denotes the join in $M_L$. Hence $M_L$ is extremally disconnected. Let finally $S$ be a further dense extremally disconnected sublocale and let $\nu_S$ denote the frame surjection associated to $S$. We need to show that $S\subseteq M_L$. Let $a\in L$. Then for any $b\in S$,
$$(a^*\vee a^{**}) \to b = \nu_S(a^*\vee a^{**})\to b=(\nu_S(a)^*\vee^S \nu_S(a)^{**})\to b=1\to b=b.$$
where the first equality follows from \eqref{LocM}, the second equality follows from the lemma above and the fact that $\nu_S$  is a dense surjection; and the third equality holds because $S$ is extremally disconnected. Hence, $S\subseteq \op(a^*\vee a^{**})$ and so $S\subseteq M_L$.
\end{proof}

In view of the previous proposition, the sublocale $M_L$ of $L$ will be referred to as the \emph{DeMorganization} of $L$.

\begin{corollary}
The DeMorganization of a locale $L$ is always a fitted sublocale of $L$.
\end{corollary}

Evidently,   the inclusion $$B_L\subseteq M_L$$ always holds
(recall that $B_L$ is the least dense sublocale). It is clear that this inclusion may be strict because in any extremally disconnected locale  (or space) one has $M_L=L$.
 By contrast, for a large class of spaces, the DeMorganization is pointless. The following is  easy to check:

\begin{lemma}
Let $L=\Omega(X)$ where $X$ is a sober $T_1$-space without isolated points. The following are equivalent
\begin{enumerate} 
\item $M_L$ is pointless;
\item Every $x\in X$ is contained the boundary of a regular open set.
\end{enumerate}
\end{lemma}

\begin{remark}
Any metric space without isolated points satisfies the assumption of this lemma.  In fact, this will follow from Corollary~\ref{last.cor.nw} below. However, there are other classes of spaces which satisfy the condition, such as the class   of locally connected spaces, see \cite{Yang}.
\end{remark}

%
%Now, the most interesting question in this setting is the following:\\[2mm]
%\emph{When do the DeMorganization and Booleanization of a locale coincide?}\\[2mm]
%Following Johnstone \cite{johtalk} we introduce the following terminology:
%
%\begin{definition} A locale $L$ is \emph{DM-averse} if $B_L=M_L$, that is, if its Booleanization coincides with its DeMorganization.
%\end{definition}

\section{The DeMorganization of metric locales}\label{sect5}

It is a well-known fact in point-set topology that a metrizable extremally disconnected space must be discrete (see e.g. \cite[14N(2)]{gilman}). Since metrizability is inherited by sublocales, this motivates to study whether the DeMorganization of a metrizable space (resp. locale) is Boolean. Although a positive result in this direction was announced in the abstract of \cite{johtalk} (for the spatial case), it did not materialize during the talk. In this section, we provide a positive answer in the more general context of metric locales.

We shall make repeated use of the following lemma. 

\begin{lemma}\label{baselemma}
Let $L$ be a regular locale without isolated points, and let  $a,b\in L$ with $a^*=0$ and $b\ne0$.  Then, there are $0\ne c,d\leq b$ with $c\wedge d=0$ and $a\vee c^*=1=a\vee d^*$.
\end{lemma}

\begin{proof}
The Booleanization of $\newdownarrow b$  contains an element $x$ such that  $0<x<b$ (for otherwise $\newdownarrow b$ would be an irreducible  \cite[Cor.~3.4\,(c)]{IRRED} and regular frame, hence by \cite[Prop.~3.2]{IRRED0} isomorphic to $\{0,1\}$, and so $\newdownarrow b$ would be an isolated point in $L$). 
Now, note that $a\wedge x \ne 0 \ne a\wedge x^*\wedge b$ (otherwise, $x\leq a^*=0$, which is a contradiction, or $x^*\wedge b\leq a^*=0$, but $  x^* \wedge b$ is the pseudocomplement of $x$ in $\newdownarrow b$, and so $x=b$ which is also a contradiction).
By regularity, $0\ne a\wedge x= \tbigvee\{z\in L \mid z\prec a\wedge x\}$ and $0\ne a\wedge x^*\wedge b= \tbigvee\{z\in L \mid z\prec a\wedge x^*\wedge b\}$ so there are $c,d\ne 0$ with $c\prec a\wedge x$ and $d\prec a\wedge x^*\wedge b$, which clearly satisfy the required conditions.
\end{proof}

The  construction in the following proposition is the key step to show that in the class of metrizable locales without isolated points the DeMorganization coincides with the Booleanization.
The idea of the proof is based on \cite[Thm.~3]{Yang}, although several modifications are needed in order to generalize it to the point-free setting via diameters. 

Given a dense element \(a\) in a metric locale \(L\), the goal of the proof is to construct  elements \(g, h \in L\) satisfying \(g^{**}\vee h^{**} \le a\) and such that the union $g^{**}\vee h^{**}$ is still dense (this is essentially the content of Proposition~\ref{mainmetriclocp} below). To achieve this, in Proposition~\ref{prop.construction} we inductively build a family \(\{B_n\}_{n\in\N}\) of subsets of \(L\) that are pairwise disjoint in a strong sense, with the diameters of elements in \(B_n\) decreasing as \(n\) grows. Within each \(b\in B_n\), we extract two disjoint elements \(g_b\) and \(h_b\), each well-inside \(a\). Thanks to the strong separation properties of \(B_n\), the joins  
\[
  g :=\; \tbigvee_{n\in \N}\tbigvee_{b\in B_n} g_b
  \quad\text{and}\quad
  h :=\; \tbigvee_{n\in \N}\tbigvee_{b\in B_n} h_b
\]
still satisfy \(g^{**} \le a\) and \(h^{**} \le a\) and their union remains dense.

 Before proceeding, we recall here that a family $\{a_i\}_{i\in I}$ in a locale is said to be \emph{discrete} if there is a cover $C$ of $L$ such that any $c\in C$ meets $a_i$ for at most one $i\in I$ (cf. \cite[Remark~4.2]{hedgehog}). If $\{a_i\}_{i\in I}$ is discrete and $b\in L$, then $b\vee \tbigwedge_i a_{i}^*= \tbigwedge_{i}(b\vee a_{i}^*)$ (see \cite[Remarks~6.1\,(2)--(3)]{hedgehog}).

\begin{proposition}\label{prop.construction}
Let $(L,d)$ be a metric locale without isolated points and $a <1 $ with $a^*=0$. Then there is a countable family $\{B_n\}_{n\in\N}$ of subsets of $L$, and for any $n\in \N$ and $b\in B_n$ there are two elements $0\ne g_{b},h_b\leq b$ with $g_b\wedge h_b=0$ satisfying the following properties for each $n\in\N$:
\begin{enumerate}[label=\textup{(\arabic*$_n$)},leftmargin=2.0\parindent]
\item \label{1n} For any $b\in B_n$,  the relation $d(b)\leq \frac{2}{n}$ holds,
\item \label{2n} For any $x\in L$ such that $d(x)\leq \frac{1}{n}$, one has $0\ne x\wedge b$ for at most one $b\in B_n$,
\item \label{3n} For any $m<n$, any $b\in  B_n$ and $b'\in B_m$, one has $g_{b'}\wedge b=0= h_{b'}\wedge b$,
\item \label{4n} $B_n$ is maximal among families satisfying \textup{(1$_n$)--(3$_n$)},
\item \label{5n} For any $b\in B_n$, the relation $a\vee g_{b}^*=1=a\vee h_{b}^*$ holds. 
\end{enumerate}
\end{proposition}

\begin{proof}
We construct the required families by strong induction.  Let $n\in\N$ and suppose  we have constructed $B_m$,  $g_b$ and $h_b$ for every $b\in B_m$ and for every $ m < n$. We now proceed to construct $B_n$.  For each $m< n$, by (2$_m$) and property \ref{D4} it follows readily  that  each $B_m$ is a discrete family, and since $g_b,h_b\leq b$, so are the families $\{g_b\}_{b\in B_m}$ and  $\{h_b\}_{b\in B_m}$. By the comment preceding the statement, it follows that $a\vee \tbigwedge_{b\in B_m} g_{b}^{*}= \tbigwedge_{b\in B_m} (a\vee g_{b}^*)=1$, and similarly $a\vee \tbigwedge_{b\in B_m} h_{b}^{*}=1$. Now, we claim that $$b:= \tbigwedge_{m<n} \tbigwedge_{b\in B_m} g_{b}^*\wedge h_{b}^* \ne 0,$$
for otherwise, $a=\tbigwedge_{ m<n} \tbigwedge_{b\in B_m}(a\vee g_{b}^*)\wedge (a\vee h_{b}^*) =1$, a contradiction. 
Since a metric locale is regular, there is an $x\ne 0$ such that $x^*\vee b=1$. Moreover, since $\{ u\mid d(u)\leq \frac{2}{n} \}$ covers $L$ by \ref{D4},  there is a $u\in L$ with $d(u)\leq \frac{2}{n}$ such that $b':=u\wedge x\ne 0$. 	We claim that the one-element family $\{b'\}$ satisfies (1$_n$), (2$_n$) and (3$_n$). Indeed, (1$_n$) and  (2$_n$) are trivial, and (3$_n$) holds because $b'^*\vee b=1$. Now, we have just seen that the family
$$\mathcal{B}_n=\{ B\subseteq L \mid  B \text{ satisfies (1$_n$)--(3$_n$)}\}$$ 
is nonempty, and hence an easy application of Zorn's Lemma yields a maximal $B_n$. Hence $B_n$ satisfies (4$_n$). The existence of the elements $g_b$ and $h_b$ satisfying (5$_n$) for $b\in B_n$ follows at once from Lemma~\ref{baselemma}.
\end{proof}

We shall need the following easy property about metric locales:

\begin{lemma}\label{cor.nee}
Let $(L,d)$ a metric locale and $b\ne 0$. Then there exists $0\ne c\leq b$ and $n\in\N$ such that for all $x\in L$ with $d(x)<\frac{1}{n}$, either $x\wedge b^* =0 $ or $x\wedge c= 0$.
\end{lemma}

\begin{proof} Since $b\ne 0$, by admissibility there is an $n\in \N$ and a $c\ne 0$ such that $c\lhd_{1/n} b$.  The required property follows immediately from the definition of the relation $\lhd$.\end{proof}

\begin{proposition}\label{mainmetriclocp}
Let $(L,d)$ be a metric locale without isolated points and $a <1 $ with $a^*=0$.  Let $\{B_n\}_{n\in\N}$, $g_b$ and $h_b$ (for $b\in B_n$ and $n\in\N$) be as in Proposition~\ref{prop.construction}, and define
$$g:=\tbigvee_{n\in\N}\tbigvee_{b\in B_n} g_b\qquad \text{ and } \qquad h:= \tbigvee_{n\in\N}\tbigvee_{b\in B_n}h_b.$$
Then, the following properties hold:
\begin{enumerate}
\item \label{mainmetriclocp1} $g\wedge h=0$,
\item \label{mainmetriclocp2} $g^*\wedge h^*=0$,
\item \label{mainmetriclocp3} $h^*\leq a \vee g^*$, and symmetrically $g^*\leq a \vee h^*$.
\end{enumerate}
\end{proposition}

\begin{proof}
(1) First, for $n\in\N $ and  $b\in B_n$, we have $g_b\wedge h_b=0$ by construction. Moreover, it follows from  (2$_n$) and \ref{D4} that  $B_n$ is a discrete family and so it is pairwise disjoint \cite[Remark~4.2]{hedgehog}. Since $g_b\leq b$ and $h_{b'}\leq b'$ one therefore has $g_b\wedge h_{b'}=0$ for $b\ne b'$ in $B_n$. If $n<m$, $b\in B_n$ and $b'\in B_m$, one has $g_b\wedge h_{b'}\leq g_b\wedge b'=0$ by property (3$_n$) and similarly $g_{b'}\wedge h_{b}=0$.\\[2mm]
(2) By way of contradiction, suppose $g^*\wedge h^*\ne 0$. By Lemma~\ref{cor.nee}, there exists a $0\ne c\leq g^*\wedge h^*$ and an $n\in\N$ such that for all $y\in L$ such that $d(y)<\frac{1}{n}$, either $y\wedge (g^*\wedge h^*)^*=0$ or $y\wedge c=0$. Observe that $y\wedge (g^*\wedge h^*)^*=0$ iff $y\leq (g^*\wedge h^*)^{**}=g^*\wedge h^*$ because double pseudocomplement commutes with finite meets. Hence, 
\begin{equation}\label{eqqq}
d(y)<\frac{1}{n} \text{ implies either } y\wedge c=0\text{ or } (y\wedge g=0 \text{ and } y\wedge h=0).
\end{equation}

Now, by \ref{D4} select a $0<d\leq c$ with $d(d)<\frac{1}{8n}$. We first note that $d\not\in B_{8n}$, for otherwise, $g_d=g_d\wedge d \leq g\wedge c=0$, a contradiction.  Moreover, the family $B_{8n}\cup \{ d\}$ also satisfies conditions (1$_{8n}$)--(3$_{8n}$). In fact, (1$_{8n}$) is trivial, and (3$_{8n}$) follows easily as for any $m<8n$ and $b\in B_m$, one has $g_{b}\wedge d\leq g\wedge c=0$ and similarly $h_{b}\wedge d=0$. The only part remaining is to show (2$_{8n}$). Hence suppose there is a $x\in L$ with $d(x)\leq \frac{1}{8n}$ and  a $b\in B_{8n}$ with $x\wedge b\ne 0 \ne x\wedge d$. By \ref{D3} one has
$$d(x\vee b)\leq d(x)+d(b) \leq \frac{1}{8n}+\frac{2}{8n} =\frac{3}{8n}$$
and 
$$d(x\vee d)\leq d(x)+d(d) < \frac{1}{8n}+\frac{1}{8n} = \frac{2}{8n}.$$
Moreover, since $(x\vee b)\wedge (x\vee d)\geq x \ne 0$, one has
$$d (x\vee b\vee d)\leq  d(x\vee b)+d(x\vee d) < \frac{5}{8n}<\frac{1}{n}.$$
But then, if we set $y:=x\vee b \vee d$, one has $d(y)<1/n$ and $y\wedge c\geq d\wedge c\ne 0$ and $y\wedge g\geq b\wedge g_{b}=g_b\ne 0$. This is in contradiction with \eqref{eqqq}, so (2$_{8n}$) is satisfied. But then the family $B_{8n}\cup\{d\}$ satisfies (1$_{8n}$)--(3$_{8n}$) which contradicts the maximality given by (4$_{8n}$).\\[2mm]
(3) We only prove the first inequality. By way of contradiction, suppose that $h^*\not\leq a \vee g^*$. By admissibility, there is an $n\in \N$ and a $c\in L$ with $c\wedge h=0$, $c\not\leq a\vee g^*$ and the property that 
\begin{equation}\label{eq1}
d(x)<\frac{1}{n} \text{ implies either } x\wedge c=0\text{ or }x\wedge h=0.
\end{equation}
Now, consider the element $b_0:= c\wedge \tbigwedge_{m \leq 2n} \tbigwedge_{b\in B_m} g_{b}^* $.
We claim $b_0\ne 0$. Otherwise, $a= (c\vee a)\wedge \tbigwedge_{ m \leq 2n}\tbigwedge_{b\in B_m}(a\vee g_{b}^*)=c\vee a$ by property (5$_{m}$) and by discreteness of the family $B_m$. Therefore, one has $c\leq a\leq a\vee g^*$, which is a contradiction.  Hence $b_0\ne 0$. By Lemma~\ref{cor.nee}, there is a $0\ne c_0\leq b_0$ and $m\in\N$ such that for all $y\in L$ such that $d(y)<\frac{1}{m}$, either $y\wedge b_{0}^*=0$ or $x\wedge c_0=0$. 
We claim that there is a $k> 2n$ and $b\in B_k$ such that $g_b\wedge c_0\ne 0$.  Indeed, if for all $k> 2n$ and $b\in B_k$ one has $g_b\wedge c_0=0$, since $c_0\leq b_0\leq\tbigwedge_{ m \leq 2n} \tbigwedge_{b\in B_m}g_{b}^*$, we also have $g_b\wedge c_0=0$ for all $k\leq 2n$ and  all $b\in B_k$.  It follows that $c_0\wedge g=0$. Then $c_{0}\leq g^*$, but $c_0\leq b_{0}\leq c\leq h^*$, from which follows  that $c_0\leq g^*\wedge h^*=0$ by item (2), a contradiction.

Therefore, there is a $k> 2n$ and $b\in B_k$ with $g_b\wedge c_0\ne 0$. Finally observe that
$$d(g_b\vee h_b)\leq d(b)\leq \frac{2}{k}<\frac{1}{n}.$$
This contradicts \eqref{eq1}, as $(g_b\vee h_b)\wedge c\geq g_b\wedge c_0\ne 0$ and $(g_b\vee h_b)\wedge h\geq h_b\ne 0$.
\end{proof}

\begin{corollary}\label{last.cor.nw}
Let $(L,d)$ be a metric locale without isolated points and let $a\in L$ with $a^*=0$. Then there is a $g\in L$ with $g^*\vee g^{**}\leq a$.
\end{corollary}

\begin{proof}
If $a=1$, there is nothing to prove, so assume $a<1$ and let $g,h\in L$ as in Proposition~\ref{mainmetriclocp}. Combining the first inequality of Proposition~\ref{mainmetriclocp}\,(\ref{mainmetriclocp3}) and Proposition~\ref{mainmetriclocp}\,(\ref{mainmetriclocp2}), it follows that $h^*\leq a$. Similarly, $g^*\leq a$. Finally, by Proposition~\ref{mainmetriclocp}\,(\ref{mainmetriclocp1}) it follows that $g^{**}\leq h^*$, which yields the required conclusion.
\end{proof}

\begin{corollary}\label{main.corollary}
In any metric locale without isolated points, the Booleanization coincides with the DeMorganization.
\end{corollary}

\begin{proof}
We only need to check the inclusion $M_L\subseteq B_L$. By Proposition~\ref{boolean1} this amounts to showing that for any dense $a\in L$, the relation $M_L\subseteq \op(a)$ holds. But by Corollary~\ref{last.cor.nw} there is a $g\in L$ with $g^*\vee g^{**}\leq a$.  It follows that $M_L\subseteq \op(g^*\vee g^{**})\subseteq \op(a)$.
\end{proof}

\begin{corollary}
An extremally disconnected metric locale  without isolated points is Boolean.
\end{corollary}

\begin{remarks}
(1) The restriction of Corollary~\ref{last.cor.nw} to the spatial case also follows easily from \cite[Lemma~1.2]{PW}, where it is proved that every closed nowhere dense subspace of a metric space can be expressed as the intersection of the closures of two disjoint open sets.
However, we point out that the scope of our result is substantially more general, as metric locales are not necessarily spatial (and in fact they may not have any points at all).\\[2mm]
(2) We do not know whether the additional condition on not having isolated points can be removed. \\[2mm]
(3) Whereas the results in Sections~\ref{sect3} and \ref{sect4} are constructively valid, the results in this section are non-constructive as we have freely used the law of excluded middle. Additionally, the proof of Proposition~\ref{prop.construction} uses Zorn's Lemma. Developing a constructive version of these results is left for future work (cf. the constructive theory of metric locales developed in \cite{metric}).\end{remarks}

\end{document}